\documentclass[12pt]{amsart}

\usepackage[cal=cm,scr=euler]{mathalfa}
\usepackage{microtype}
\usepackage{mlmodern}

\usepackage[T1]{fontenc}
\usepackage[a4paper,margin=2.5cm,top=2.5cm,bottom=2.5cm,centering,headheight=3ex,headsep=4ex,vcentering]{geometry}
\usepackage{amsmath,amssymb,amsthm,mathtools}
\usepackage{xcolor}
\usepackage[colorlinks=true,linkcolor=blue,citecolor=blue,urlcolor=blue]{hyperref}

\newtheorem{theorem}{Theorem}[section]

\newtheorem{lemma}[theorem]{Lemma}
\newtheorem{corollary}[theorem]{Corollary}
\newtheorem{conjecture}[theorem]{Conjecture}
\theoremstyle{remark}
\newtheorem{remark}[theorem]{Remark}

\newcommand{\CT}{\operatorname{CT}}
\newcommand{\Z}{\mathbb{Z}}
\newcommand{\cphi}{c\phi}
\newcommand{\CPhi}{\mathrm{C}\Phi}

\title{Generalized Frobenius Partitions Modulo Powers of $2$}
\author{Manjil P. Saikia}
\address{Mathematical and Physical Sciences division, School of Arts and Sciences, Ahmedabad University, Navrangpura, Ahmedabad 380009, Gujarat, India}
\email{manjil.saikia@ahduni.edu.in}



\begin{document}

\begin{abstract}
Let $c\phi_k(n)$ denote the number of $k$-colored generalized Frobenius partitions of $n$.  We prove that, for every $m\geq2$ and every $k\equiv2\pmod{2^m}$,
\[
 \sum_{n\geq0}c\phi_k(n)q^n\equiv\frac{\varphi(q)\,(q^2;q^2)_\infty}{(q;q)_\infty^2}\sum_{n\geq0}c\phi_{k/2}(n)q^{2n}\pmod{2^m},
\]
where $\varphi(q)$ is the classical theta function.  For $m=2$ this recovers a congruence of Chan, Wang, and Yang. Applied with $k=18$, we determine $c\phi_{18}(2n+1)$ modulo $16$ completely. In particular, we also prove
\[
 \sum_{n\geq0}c\phi_{18}(6n+1)q^n
 \equiv4\sum_{r\in\mathbb{Z}}q^{r(3r-1)/2}\pmod{16},
\]
which proves the congruences $c\phi_{18}(30n+19)\equiv c\phi_{18}(30n+25)\equiv0\pmod{16}$ recently conjectured by Das, Nath, and Sarma (2026). It also yields further congruences modulo $16$ and a simple modulo-$8$
characterization that recovers and extends a recent congruence of those authors.
As a second application we set $k=10$ and determine $c\phi_{10}(2n+1)$ modulo $8$.
\end{abstract}

\maketitle

\section{Introduction}

A generalized Frobenius partition of $n$ is a two-rowed array
\[
 \begin{pmatrix}
  a_1 & a_2 & \cdots & a_r\\
  b_1 & b_2 & \cdots & b_r
 \end{pmatrix}
\]
of nonnegative integers with $n=r+\sum_{i}a_i+\sum_{i}b_i$, where
the entries have a prescribed ordering condition.  When both
rows are strictly decreasing, these arrays are the classical
Frobenius symbols and are in bijection with ordinary partitions of
$n$.  In his 1984 Memoir \cite{Andrews1984}, Andrews studied several
other ordering conditions, the most influential of which is the
following: the entries are taken from $k$ copies of the nonnegative
integers, distinguished by $k$ colors, and each row consists of
distinct colored integers, strictly decreasing in a fixed total order
on the colored integers.  The resulting objects are the $k$-colored
generalized Frobenius partitions.  We write $\cphi_k(n)$ for their
number and
\[
 \CPhi_k(q):=\sum_{n=0}^{\infty}\cphi_k(n)q^n
\]
for the generating function. Thus $\cphi_1(n)=p(n)$ is the ordinary
partition function.  Andrews showed that $\CPhi_k(q)$ is the constant
term in $z$ of an explicit infinite product (recalled in
\eqref{eq:Andrews-CT} below) and used this, among other things, to
prove the Ramanujan-type congruence $\cphi_2(5n+3)\equiv0\pmod5$.

The arithmetic of $\cphi_k(n)$ has since been studied intensively, so we mention only some highlights.
Kolitsch \cite{Kolitsch1989,Kolitsch1991} introduced the companion
function $\overline{\cphi}_k(n)$,
proved $\overline{\cphi}_k(n)\equiv0\pmod{k^2}$, and related
$\overline{\cphi}_k$ to ordinary partitions for $k=5,7,11$; Sellers
\cite{Sellers1993,Sellers1994} sharpened Kolitsch's congruence on the
progressions $kn$ for $k=2,3,5,7,11$.  The Andrews--Sellers conjecture on
$\cphi_2$ modulo powers of $5$ \cite{Andrews1984,Sellers1994b} was proved
by Paule and Radu \cite{PauleRadu2012} using modular functions. Baruah and Sarmah \cite{BaruahSarmah2011,BaruahSarmah2015} studied
$\cphi_4$ and $\cphi_6$ and posed conjectures on $\cphi_6$ that were
settled by Xia \cite{Xia2015}. Hirschhorn \cite{Hirschhorn2016} subsequently gave further proofs of
congruences for $\cphi_6$, and Cui and Gu \cite{CuiGu2019} established congruences modulo
powers of $2$ for $\cphi_6$.

Garvan and Sellers
\cite{GarvanSellers2014} observed that the congruence
$(1+X)^p\equiv1+X^p\pmod p$, applied factor by factor inside Andrews'
constant-term representation, transfers any congruence
$\cphi_k(pn+r)\equiv0\pmod p$ to $\cphi_{pN+k}(pn+r)\equiv0\pmod p$
for all $N\geq0$.  A unifying framework was provided
by Chan, Wang, and Yang \cite{ChanWangYang2019}, who showed that
$\CPhi_k(q)$ is (up to a power of $q$ and an eta-quotient) a modular
form and expressed it explicitly as a linear combination of
eta-quotients for many $k$.  Using Kolitsch's congruence they also proved
\cite[Theorem~5.3]{ChanWangYang2019} that, for $p$ prime, $\alpha\geq1$ and $p\nmid N$,
\begin{equation}
 \cphi_{p^\alpha N}(n)\equiv\cphi_{p^{\alpha-1}N}(n/p)\pmod{p^{2\alpha}},
 \label{eq:CWY53}
\end{equation}
where $\cphi_k(x)=0$ for $x\notin\Z$.  Recently Cui, Gu, and Tang
\cite{CuiGuTang2024,CuiGuTang2025} developed the ``method of constant
terms'' and integer matrix exact covering systems to obtain expressions for $\CPhi_k(q)$
with integral coefficients, together with several infinite families of
congruences.

For $k$ even, \eqref{eq:CWY53} with $p=2$ reads
$\cphi_{k}(n)\equiv\cphi_{k/2}(n/2)\pmod{2^{2v_2(k)}}$.  When
$k\equiv2\pmod4$ this gives information modulo $4$ only, namely
$\cphi_k(2n+1)\equiv0$ and $\cphi_k(2n)\equiv\cphi_{k/2}(n)\pmod4$; see
\cite[(6.7), (6.22), (6.33)]{ChanWangYang2019} for $k=6,10,14$.  The
first purpose of this paper is to show that an elementary refinement of
the Garvan--Sellers idea goes beyond modulus $4$ for such $k$.  

To state our first result, we need the following standard notations
\[
 f_r:=(q^r;q^r)_\infty:=\prod_{i\geq 0}(1-q^{r+ri})\qquad \text{and} \qquad 
 \varphi(q):=\sum_{n\in\Z}q^{n^2}.
\]

\begin{theorem}\label{thm:general}
Let $m\geq2$ and let $k\equiv2\pmod{2^m}$.  Then
\begin{equation}
 \CPhi_k(q)\equiv\frac{\varphi(q)f_2}{f_1^2}\,\CPhi_{k/2}(q^2)\pmod{2^m}.
 \label{eq:general}
\end{equation}
Consequently,
\begin{equation}
 \sum_{n\geq0}\cphi_k(2n+1)q^n\equiv4\frac{f_4^4}{f_1^4}\CPhi_{k/2}(q)
 \quad\text{and}\quad
 \sum_{n\geq0}\cphi_k(2n)q^n\equiv\frac{f_2^{12}}{f_1^8f_4^4}\CPhi_{k/2}(q)
 \pmod{2^m}.
 \label{eq:general-parts}
\end{equation}
\end{theorem}

Since $\varphi(q)f_2/f_1^2\equiv1\pmod4$ (see Lemma~\ref{lem:binomial}),
the case $m=2$ of \eqref{eq:general} is exactly \eqref{eq:CWY53} with
$(p,\alpha)=(2,1)$; the cases $m\geq3$ appear to be new.

Our main application is to $k=18$, where Theorem~\ref{thm:general}
applies with $m=4$. In fact, the results we now describe were the motivation and the starting point of this paper.  Das, Nath, and Sarma
\cite{DasNathSarma2026} recently proved the congruence $\cphi_{18}(3n+2)\equiv 0 \pmod{3^7}$
stated by Cui, Gu, and Tang \cite{CuiGuTang2025} as a conjecture, and, in the same
paper they \cite{DasNathSarma2026}, proposed the following.

\begin{conjecture}[{\cite[Conjecture~5.1]{DasNathSarma2026}}]
\label{conj:DNS}
For every $n\geq0$,
\begin{equation}
 \cphi_{18}(30n+19)\equiv0\pmod{16},\qquad
 \cphi_{18}(30n+25)\equiv0\pmod{16}.
 \label{eq:conjecture}
\end{equation}
\end{conjecture}

We prove Conjecture~\ref{conj:DNS} by determining $\cphi_{18}(2n+1)$
modulo $16$ completely. We need the following standard notation
\[
 \psi(q):=\sum_{n\geq0}q^{n(n+1)/2}.
\]

\begin{theorem}\label{thm:odd18}
We have
\begin{equation}
 \sum_{n=0}^{\infty}\cphi_{18}(2n+1)q^n
 \equiv4\psi(q)+8q^4\psi(q)\frac{f_9^{12}}{f_3^4}\pmod{16}.
 \label{eq:odd18}
\end{equation}
\end{theorem}

Extracting the terms $q^{3n}$ from \eqref{eq:odd18} gives the following.

\begin{theorem}\label{thm:strong}
We have
\begin{equation}
 \sum_{n=0}^{\infty}\cphi_{18}(6n+1)q^n
 \equiv4\sum_{r\in\Z}q^{r(3r-1)/2}\pmod{16}.
 \label{eq:strong}
\end{equation}
\end{theorem}

Theorem~\ref{thm:strong} immediately implies \eqref{eq:conjecture},
but it says considerably more.  Since the generalized pentagonal
numbers $r(3r-1)/2$, $r\in\Z$, are pairwise distinct, the coefficients
of the right side of \eqref{eq:strong} are all $0$ or $4$, so
\eqref{eq:strong} determines $\cphi_{18}(6n+1)$ modulo $16$
completely.  Furthermore, as $24\cdot\frac{r(3r-1)}{2}+1=(6r-1)^2$ is
always a perfect square, we obtain the following.

\begin{corollary}\label{cor:sharp}
For every $n\geq0$,
\[
 \cphi_{18}(6n+1)\equiv
 \begin{cases}
  4\pmod{16},& \text{if } 24n+1 \text{ is a perfect square},\\
  0\pmod{16},& \text{otherwise}.
 \end{cases}
\]
In particular, if $p\geq5$ is prime and $24n+1$ is a quadratic
non-residue modulo $p$, then $\cphi_{18}(6n+1)\equiv0\pmod{16}$.
Consequently, for every $n\geq0$,
\begin{align}
 \cphi_{18}(30n+19)&\equiv\cphi_{18}(30n+25)\equiv0\pmod{16},
 \label{eq:cor-p5}\\
 \cphi_{18}(42n+19)&\equiv\cphi_{18}(42n+25)\equiv\cphi_{18}(42n+37)
 \equiv0\pmod{16}.
 \label{eq:cor-p7}
\end{align}
\end{corollary}

The congruences \eqref{eq:cor-p5} are exactly \eqref{eq:conjecture}. The modulus $16$ in Theorem~\ref{thm:strong} is best possible on the
progression $6n+1$, since
\[
 \cphi_{18}(19)=1201156217053086096\equiv16\pmod{32},
\]
whereas the coefficient of $q^3$ in
$4\sum_{r\in\Z}q^{r(3r-1)/2}$ is $0$.
Reading \eqref{eq:odd18}
modulo $8$ instead gives a second family of congruences.

\begin{corollary}\label{cor:mod8}
For every $n\geq0$,
\[
 \cphi_{18}(2n+1)\equiv
 \begin{cases}
  4\pmod{8},& \text{if } 8n+1 \text{ is a perfect square},\\
  0\pmod{8},& \text{otherwise}.
 \end{cases}
\]
In particular, for every $n\geq0$,
\begin{equation}
 \cphi_{18}(6n+5)\equiv\cphi_{18}(10n+5)\equiv\cphi_{18}(10n+9)\equiv0\pmod8.
 \label{eq:cor-mod8}
\end{equation}
\end{corollary}

The congruence $\cphi_{18}(6n+5)\equiv0\pmod8$ in
\eqref{eq:cor-mod8} was proved independently by Das, Nath, and Sarma
\cite[Theorem~1.2]{DasNathSarma2026}.  Their proof gives a substantially
longer $q$-product expression for the odd part of $\CPhi_{18}(q)$ modulo
$8$. The two other progressions in
\eqref{eq:cor-mod8} appear to be new.

As a second illustration of Theorem~\ref{thm:general} we treat $k=10$
with $m=3$, where \eqref{eq:CWY53} gives only $\cphi_{10}(2n+1)\equiv0\pmod4$.

\begin{theorem}\label{thm:ten}
We have
\begin{equation}
 \sum_{n=0}^{\infty}\cphi_{10}(2n+1)q^n\equiv4\frac{f_8}{f_1}\pmod8.
 \label{eq:ten}
\end{equation}
\end{theorem}

Our proofs are elementary.  Theorem~\ref{thm:general} rests on the
congruence $(1+X)^{2^m}\equiv(1+X^2)^{2^{m-1}}\pmod{2^m}$ applied factor by
factor inside Andrews' constant-term representation, together with an
exact evaluation of the even part in $z$ of the square of the Jacobi theta function.
For $k=18$, the exact formula of Chan, Wang, and Yang for $\CPhi_9$
collapses modulo $4$ to a two-term expression involving the Borwein
cubic theta function $a(q)$ (Lemma~\ref{lem:C9mod4}); combined with the
observation $f_1^3\equiv\psi(q)\pmod4$ this yields Theorem~\ref{thm:odd18},
and the pentagonal theta series then appears through the classical
$3$-dissection of $\psi(q)$.  For $k=10$ the same constant-term argument
gives $\CPhi_5(q)$ modulo $2$.

The rest of the paper is organized as follows: Section~\ref{sec:prelim} collects the $q$-series identities we need,
Section~\ref{sec:reductions} proves Theorem~\ref{thm:general},
Section~\ref{proof} proves Theorems~\ref{thm:odd18} and \ref{thm:strong}
and their corollaries, Section~\ref{sec:ten} proves Theorem~\ref{thm:ten}, and Section~\ref{conc} ends the paper with some ideas for future study.

\section{Preliminaries}\label{sec:prelim}

If $F(q)=\sum_{n\geq0}a(n)q^n$, define the Atkin $U$-operator
\[
 U_3F(q):=\sum_{n\geq0}a(3n)q^n.
\]
We require the classical theta function
\[
 \Theta(z,q):=\sum_{r\in\Z}z^r q^{r(r+1)/2},
\]
and the Borwein cubic theta function
\[
 a(q):=\sum_{m,n\in\Z}q^{m^2+mn+n^2}.
\]
Jacobi's triple-product identity gives
\begin{equation}
 \Theta(z,q)=f_1(-zq;q)_\infty(-z^{-1};q)_\infty,
 \label{eq:JTP-theta}
\end{equation}
and the product representations
\begin{equation}
 \varphi(q)=\frac{f_2^5}{f_1^2f_4^2},\qquad
 \psi(q)=\frac{f_2^2}{f_1}.
 \label{eq:phi-product}
\end{equation}

We will also use the standard identities
\begin{align}
 \frac{1}{f_1^4}
 &=\frac{f_4^{14}}{f_2^{14}f_8^4}
   +4q\frac{f_4^2f_8^4}{f_2^{10}},
 \label{eq:two-dissection}\\
 f_1^3&=a(q^3)f_3-3qf_9^3,
 \label{eq:cubic-dissection}\\
 \psi(q)&=\frac{f_6f_9^2}{f_3f_{18}}+q\frac{f_{18}^2}{f_9}.
 \label{eq:psi-3dissection}
\end{align}
Identity \eqref{eq:two-dissection} is the usual 2-dissection of
$1/f_1^4$; see \cite[Lemma~2]{BaruahDas2022}.
Identity \eqref{eq:cubic-dissection} is equivalent to the
$3$-dissection $b(q)=a(q^3)-c(q^3)$ of the Borwein cubic theta
function $b(q)=f_1^3/f_3$, together with the product representation
$c(q)=3q^{1/3}f_3^3/f_1$; see \cite[(22.1.5)--(22.1.7)]{Hirschhorn2017}. Identity \eqref{eq:psi-3dissection}
is the $3$-dissection of $\psi(q)$; see \cite[(14.3.3)]{Hirschhorn2017}. Finally, since $(-1)^r(2r+1)\equiv1\pmod4$ for every integer $r\geq0$,
Jacobi's identity $f_1^3=\sum_{r\geq0}(-1)^r(2r+1)q^{r(r+1)/2}$ gives
\begin{equation}
 f_1^3\equiv\psi(q)\pmod4.
 \label{eq:f13-psi}
\end{equation}

We also need the following easy to prove lemma.
\begin{lemma}\label{lem:binomial}
For every $m\geq1$ the following formal congruences hold:
\begin{align}
 (1+X)^{2^m}&\equiv(1+X^2)^{2^{m-1}}\pmod{2^m},\label{eq:bin2m}\\
 f_r^{2^m}&\equiv f_{2r}^{2^{m-1}}\pmod{2^m}.\label{eq:fr2m}
\end{align}
In particular, $f_r^4\equiv f_{2r}^2\pmod4$ and
\begin{equation}
 \frac{\varphi(q)f_2}{f_1^2}=\frac{f_2^6}{f_1^4f_4^2}\equiv1\pmod4.
 \label{eq:Eunit}
\end{equation}
\end{lemma}

\begin{proof}
Write
\[
 (1+X)^{2^m}=\bigl(1+X^2+2X\bigr)^{2^{m-1}}
 =\sum_{j=0}^{2^{m-1}}\binom{2^{m-1}}{j}(2X)^j(1+X^2)^{2^{m-1}-j}.
\]
We have $v_2\binom{2^{m-1}}{j}=m-1-v_2(j)$ for $1\leq j\leq2^{m-1}$,
so $v_2\bigl(\binom{2^{m-1}}{j}2^j\bigr)=m-1+j-v_2(j)\geq m$ for $j\geq1$.
Only the $j=0$ term survives modulo $2^m$, which gives \eqref{eq:bin2m}.

Replacing $X$ by $-q^{rm}$ in \eqref{eq:bin2m} and multiplying over $m\geq1$
gives \eqref{eq:fr2m}. 

Finally, we have $\varphi(q)=f_2^5/(f_1^2f_4^2)$  and $f_1^4f_4^2\equiv f_2^2\cdot f_2^4=f_2^6\pmod4$
by \eqref{eq:fr2m}, immediately giving us \eqref{eq:Eunit}.
\end{proof}

\section{Proof of Theorem \ref{thm:general}}\label{sec:reductions}

Let
\begin{equation}
 \mathcal R_k(z,q)
 =\prod_{j=0}^{\infty}(1+zq^j)^k.
 \label{eq:row-generating-function}
\end{equation}
Then, it is known (see \cite[Eq. (5.14)]{Andrews1984} or \cite[Theorem 2.1]{GarvanSellers2014}) that
\begin{equation}
 \CPhi_k(q)
 =\CT_z\!\left(
   \mathcal R_k(zq,q)\mathcal R_k(z^{-1},q)
  \right),
 \label{eq:general-principle-specialized}
\end{equation}
where $\CT_z(f(z))$ denotes the constant term (that is, the coefficient of $z^0$) in the Laurent series $f(z)$. Using \eqref{eq:row-generating-function}, the expression inside this
constant term is
\begin{align*}
 \mathcal R_k(zq,q)\mathcal R_k(z^{-1},q)
 &=\prod_{j=0}^{\infty}(1+zq^{j+1})^k
   \prod_{j=0}^{\infty}(1+z^{-1}q^j)^k\\
 &=\left(
   \prod_{m=1}^{\infty}(1+zq^m)(1+z^{-1}q^{m-1})
   \right)^k.
\end{align*}
Consequently,
\begin{equation}
 \CPhi_k(q)=\CT_z J(z,q)^k,
 \qquad
 J(z,q):=\prod_{m=1}^{\infty}
 (1+zq^m)(1+z^{-1}q^{m-1}).
 \label{eq:Andrews-CT}
\end{equation}
By \eqref{eq:JTP-theta}, we also have
\begin{equation}
 J(z,q)=\frac{\Theta(z,q)}{f_1}.
 \label{eq:J-theta}
\end{equation}

\begin{lemma}\label{lem:theta-square}
We have
\begin{equation}
 \Theta(z,q)^2=\varphi(q)\,\Theta(z^2,q^2)+2\psi(q^2)\sum_{r\in\Z}z^{2r+1}q^{(r+1)^2}.
 \label{eq:theta-square}
\end{equation}
In particular, the even-power part in $z$ of $\Theta(z,q)^2$ is
$\varphi(q)\Theta(z^2,q^2)$, and its odd-power part is divisible by $2$.
\end{lemma}

\begin{proof}
From the definition of $\Theta(z,q)$,
\[
 \Theta(z,q)^2
 =
 \sum_{a,b\in\Z}
 z^{a+b}
 q^{a(a+1)/2+b(b+1)/2}.
\]
Fix $r\in\Z$.  Every solution of $a+b=2r$ can be written uniquely as
$a=r+s$, $b=r-s$ with $s\in\Z$, and then
$a(a+1)/2+b(b+1)/2=r(r+1)+s^2$.  Hence
\[
 [z^{2r}]\Theta(z,q)^2=q^{r(r+1)}\sum_{s\in\Z}q^{s^2}=q^{r(r+1)}\varphi(q),
\]
and summing over $r$ gives the first term in \eqref{eq:theta-square}.

Similarly, every solution of $a+b=2r+1$ can be written uniquely as
$a=r+1+s$, $b=r-s$, and then $a(a+1)/2+b(b+1)/2=(r+1)^2+s(s+1)$.  Hence
\[
 [z^{2r+1}]\Theta(z,q)^2=q^{(r+1)^2}\sum_{s\in\Z}q^{s(s+1)}=2\psi(q^2)q^{(r+1)^2},
\]
which gives the second term.
\end{proof}

We can now prove Theorem \ref{thm:general}.
\begin{proof}[Proof of Theorem~\ref{thm:general}]
Write $k=2+2^mt$ with $t\geq0$ and put $N=2^{m-1}t=(k-2)/2$, so that $k/2=N+1$. Applying \eqref{eq:bin2m} separately to every factor
$(1+zq^j)^{2^mt}$ and $(1+z^{-1}q^{j-1})^{2^mt}$ in \eqref{eq:Andrews-CT} gives
\begin{equation}
     J(z,q)^{k}=J(z,q)^2J(z,q)^{2^mt}
 \equiv J(z,q)^2J(z^2,q^2)^{N}\pmod{2^m}.\label{n1}
\end{equation}
When $t=0$, both sides of this congruence are simply $J(z,q)^2$, so
the argument below also applies with $N=0$.
By \eqref{eq:J-theta},
\begin{equation}
    J(z,q)^2J(z^2,q^2)^{N}=\frac{\Theta(z,q)^2\,\Theta(z^2,q^2)^{N}}{f_1^2f_2^{N}}.\label{n2}
\end{equation}
Since $\Theta(z^2,q^2)^{N}$ contains only even powers of $z$, an odd
power of $z$ from $\Theta(z,q)^2$ cannot combine with it to produce $z^0$.
Therefore, by Lemma~\ref{lem:theta-square}, only the even-power part of
$\Theta(z,q)^2$ contributes to the constant term:
\begin{equation}
     \CT_z\bigl(\Theta(z,q)^2\Theta(z^2,q^2)^{N}\bigr)
 =\varphi(q)\,\CT_z\bigl(\Theta(z^2,q^2)^{N+1}\bigr)
 =\varphi(q)\,\CT_w\bigl(\Theta(w,q^2)^{N+1}\bigr),\label{n3}
\end{equation}
where $w=z^2$.  

By \eqref{eq:Andrews-CT} and \eqref{eq:J-theta} with $q$
replaced by $q^2$ and $k$ replaced by $N+1$,
\begin{equation}
\CT_w\bigl(\Theta(w,q^2)^{N+1}\bigr)=f_2^{N+1}\CPhi_{N+1}(q^2).\label{n4}
\end{equation}
Putting everything together, using \eqref{n1}--\eqref{n4}, we have
\[
 \CPhi_k(q)=\CT_zJ(z,q)^k\equiv\frac{\varphi(q)f_2^{N+1}}{f_1^2f_2^{N}}\CPhi_{N+1}(q^2)
 =\frac{\varphi(q)f_2}{f_1^2}\CPhi_{k/2}(q^2)\pmod{2^m},
\]
which is \eqref{eq:general}.

For \eqref{eq:general-parts}, we use \eqref{eq:phi-product} and \eqref{eq:two-dissection} in \eqref{eq:general}:
\begin{equation}
 \frac{\varphi(q)f_2}{f_1^2}
 =\frac{f_2^6}{f_1^4f_4^2}
 =\frac{f_4^{12}}{f_2^8f_8^4}
   +4q\frac{f_8^4}{f_2^4}.
 \label{eq:prefactor}
\end{equation}
The first summand contains only even powers of $q$, the second only odd
powers, and $\CPhi_{k/2}(q^2)$ contains only even powers.  Separating
\eqref{eq:general} into even and odd powers of $q$ and replacing $q^2$
by $q$ gives \eqref{eq:general-parts}.
\end{proof}

\begin{remark}\label{rem:general}
(i) For $m=2$, \eqref{eq:Eunit} shows that \eqref{eq:general} reduces to
$\CPhi_k(q)\equiv\CPhi_{k/2}(q^2)\pmod4$, that is, to \eqref{eq:CWY53} with
$p=2$, $\alpha=1$, $N=k/2$.  Thus Theorem~\ref{thm:general} is consistent
with, and for $m\geq3$ sharpens, the congruence of Chan, Wang, and Yang.

(ii) The argument does not extend to $k\equiv j\pmod{2^m}$ with
$j\neq0,2$: for $j=1$ the even part in $z$ of $\Theta(z,q)$ is
$\Theta(z^2q^{-1},q^4)$, which is not a multiple of $\Theta(z^2,q^2)$, and
for $j\geq3$ the odd part of $\Theta(z,q)^2$ in Lemma~\ref{lem:theta-square}
begins to contribute.  For $j=0$ one recovers \eqref{eq:CWY53} only modulo
$2^{m}$, which is weaker than \eqref{eq:CWY53} itself.
\end{remark}

\section{Proofs of Theorems \ref{thm:odd18} and \ref{thm:strong} and their corollaries}\label{proof}

We need an exact identity due to Chan, Wang, and Yang 
\cite[Theorem~5.1]{ChanWangYang2019}:
\begin{align}
 \CPhi_9(q)
={}&324q\frac{f_3^8}{f_1^9}
 +19683q^4\frac{f_9^{12}}{f_1^9f_3^4}
 -240q\frac{f_9^3}{f_3^4}-1458q^2\frac{f_9^6}{f_1^3f_3^4}
 +\frac{f_1^3}{f_3^4}.
 \label{eq:C9-exact}
\end{align}

\begin{lemma}\label{lem:C9mod4}
We have
\begin{equation}
 f_1^9\CPhi_9(q)
 \equiv a(q^3)^4+2q^4\frac{f_9^{12}}{f_3^4} \pmod 4.
 \label{eq:C9mod4}
\end{equation}
\end{lemma}

\begin{proof}
Multiplying \eqref{eq:C9-exact} by $f_1^9$ and reducing the
coefficients modulo $4$ gives
\begin{equation}
 f_1^9\CPhi_9(q)
 \equiv
 \frac{f_1^{12}+2q^2f_1^6f_9^6-q^4f_9^{12}}{f_3^4}
 \pmod4.
 \label{eq:C9-intermediate}
\end{equation}
Put $X=f_1^3$, $Y=qf_9^3$, and $Z=a(q^3)f_3$, so that
\eqref{eq:cubic-dissection} reads $X=Z-3Y$.  The numerator in
\eqref{eq:C9-intermediate} is $X^4+2X^2Y^2-Y^4$, and
\begin{align*}
 (Z-3Y)^4+2Y^2(Z-3Y)^2-Y^4
& \equiv Z^4+2Y^4\pmod4.
\end{align*}
Substituting back and dividing by $f_3^4$ proves \eqref{eq:C9mod4}.
\end{proof}

\begin{remark}\label{rem:CGT}
Lemma~\ref{lem:C9mod4} can also be derived from the expression
\[
 \CPhi_9(q)=\frac{1}{f_1^9}\Bigl(a(q)^3a(q^3)+54q\,a(q)\frac{f_3^9}{f_1^3}+162q^2\frac{f_3^8f_9^3}{f_1^3}\Bigr)
\]
of Cui, Gu, and Tang \cite[Theorem~3.5]{CuiGuTang2024}, using
$a(q)=a(q^3)+6qf_9^3/f_3$ \cite[(4.3)]{CuiGuTang2024} and the
$3$-dissection of $1/f_1^3$ \cite[(4.5)]{CuiGuTang2024}.
\end{remark}

Define
\begin{equation}
 H(q):=\frac{f_4^4}{f_1^4}\CPhi_9(q),
 \label{eq:H}
\end{equation}
so that the case $k=18$, $m=4$ of \eqref{eq:general-parts} reads
\begin{equation}
 \sum_{n\ge0}\cphi_{18}(2n+1)q^n\equiv4H(q)\pmod{16}.
 \label{eq:oddpart}
\end{equation}
By Lemma~\ref{lem:binomial}, $f_1^{16}\equiv f_2^8\equiv f_4^4\pmod4$,
and hence $f_4^4/f_1^{13}\equiv f_1^3\pmod4$.  Combining this with
Lemma~\ref{lem:C9mod4} and \eqref{eq:f13-psi}, we obtain
\begin{equation}
 H(q)=\frac{f_4^4}{f_1^{13}}\cdot f_1^9\CPhi_9(q)
 \equiv f_1^3\left(a(q^3)^4+2q^4\frac{f_9^{12}}{f_3^4}\right)
 \equiv \psi(q)\,a(q^3)^4+2q^4\psi(q)\frac{f_9^{12}}{f_3^4}
 \pmod4.
 \label{eq:H-via-psi}
\end{equation}

\begin{proof}[Proof of Theorem \ref{thm:odd18}]
Every nonconstant coefficient of $a(q)$ is even.  Indeed,
the involution
\[
 (m,n)\longmapsto(-m,-n)
\]
preserves the quadratic form $m^2+mn+n^2$.  Its only fixed point in
$\Z^2$ is $(0,0)$, which contributes the constant term of $a(q)$.
Thus, we can write
\[
 a(q)
 =
 1+
 2\sum_{m=1}^{\infty}\sum_{n\in\Z}
 q^{m^2+mn+n^2}
 +
 2\sum_{n=1}^{\infty}q^{n^2}.
\]
Consequently, $a(q)\equiv1\pmod2$, hence $a(q)^2\equiv1\pmod4$ and
$a(q)^4\equiv1\pmod8$.  In particular $a(q^3)^4\equiv1\pmod4$, and
\eqref{eq:H-via-psi} becomes
\begin{equation}
 H(q)\equiv\psi(q)+2q^4\psi(q)\frac{f_9^{12}}{f_3^4}\pmod4.
 \label{eq:H-final}
\end{equation}
Writing $H(q)=\psi(q)+2q^4\psi(q)f_9^{12}/f_3^4+4\mathcal B(q)$ with
$\mathcal B(q)\in\Z[[q]]$ and substituting into \eqref{eq:oddpart} gives
\[
 \sum_{n\ge0}\cphi_{18}(2n+1)q^n\equiv4\psi(q)+8q^4\psi(q)\frac{f_9^{12}}{f_3^4}+16\mathcal B(q)
 \equiv4\psi(q)+8q^4\psi(q)\frac{f_9^{12}}{f_3^4}\pmod{16},
\]
which is \eqref{eq:odd18}.
\end{proof}

\begin{proof}[Proof of Theorem \ref{thm:strong}]
Define
\[
 R(q):=\frac{f_2f_3^2}{f_1f_6},
\]
so that \eqref{eq:psi-3dissection} reads
\begin{equation}
 \psi(q)=R(q^3)+q\,\psi(q^9).
 \label{eq:psi-RS}
\end{equation}
In particular, $\psi(q)$ has no terms whose exponents are congruent to
$2$ modulo $3$, and $U_3\psi(q)=R(q)$ by \eqref{eq:psi-3dissection}.

We apply $U_3$ to \eqref{eq:odd18}.  Put
\[
 F(q):=\sum_{n=0}^{\infty}\cphi_{18}(2n+1)q^n,
 \qquad\text{so that}\qquad
 U_3F(q)=\sum_{n=0}^{\infty}\cphi_{18}(6n+1)q^n.
\]
Equation \eqref{eq:odd18} says $F(q)=4\psi(q)+8q^4\psi(q)f_9^{12}/f_3^4+16\mathcal A(q)$
for some $\mathcal A(q)\in\Z[[q]]$, and $U_3$ is $\Z$-linear, so
\begin{equation}
 U_3F(q)\equiv4U_3\psi(q)+8U_3\Bigl(q^4\psi(q)\frac{f_9^{12}}{f_3^4}\Bigr)\pmod{16}.
 \label{eq:U3F}
\end{equation}
The quotient $f_9^{12}/f_3^4$ is a series in $q^3$, so a term of $q^4\psi(q)f_9^{12}/f_3^4$ has
exponent divisible by $3$ only if it comes from a term of $\psi(q)$ with
exponent congruent to $2$ modulo $3$.  There are no such terms by
\eqref{eq:psi-RS}, so the second summand in \eqref{eq:U3F} vanishes and
\begin{equation}
 U_3F(q)=\sum_{n=0}^{\infty}\cphi_{18}(6n+1)q^n\equiv4R(q)\pmod{16}.
 \label{eq:C18-6n1-product}
\end{equation}

Finally, Jacobi's triple-product identity gives
\begin{align}
 R(q)=\frac{f_2f_3^2}{f_1f_6}
 &=f_3(-q;q^3)_\infty(-q^2;q^3)_\infty=\sum_{r\in\Z}q^{r(3r-1)/2},
 \label{eq:pentagonal-theta}
\end{align}
where, for the first equality, we use
$f_2/f_1=(-q;q)_\infty$ and $f_6/f_3=(-q^3;q^3)_\infty$, and the
second equality is Jacobi's triple product with base $q^3$.
Substituting \eqref{eq:pentagonal-theta} into \eqref{eq:C18-6n1-product}
proves \eqref{eq:strong}.
\end{proof}

\begin{proof}[Proof of Corollary~\ref{cor:sharp}]
The map $r\mapsto r(3r-1)/2$ is injective on $\Z$, so the coefficient
of $q^n$ in the sum on the right side of \eqref{eq:strong} is $1$ if
$n=r(3r-1)/2$ for some $r\in\Z$ and $0$ otherwise.  Now
$n=r(3r-1)/2$ for some $r\in\Z$ if and only if $24n+1=(6r-1)^2$, and
every odd square $m^2$ with $m\equiv\pm1\pmod6$ arises in this way;
since $24n+1\equiv1\pmod{24}$ forces any square root to be
$\equiv\pm1\pmod6$, the condition is simply that $24n+1$ is a perfect
square.  This gives us the first part of the result.

If $24n+1$ is a quadratic non-residue modulo a prime $p$, it is not a
square, so $\cphi_{18}(6n+1)\equiv0\pmod{16}$.

For $p=5$ the
non-residues are $2,3\pmod5$, and $24n+1\equiv2,3\pmod 5$ exactly when
$n\equiv3,4\pmod5$, giving \eqref{eq:cor-p5} since
$6(5n+3)+1=30n+19$ and $6(5n+4)+1=30n+25$.

For $p=7$ the
non-residues are $3,5,6\pmod7$, and $24n+1\equiv3,5,6\pmod7$ exactly
when $n\equiv3,4,6\pmod7$, giving \eqref{eq:cor-p7} since
$6(7n+3)+1=42n+19$, $6(7n+4)+1=42n+25$, and $6(7n+6)+1=42n+37$.
\end{proof}

\begin{proof}[Proof of Corollary~\ref{cor:mod8}]
Reducing \eqref{eq:odd18} modulo $8$ gives
\[\sum_{n\ge0}\cphi_{18}(2n+1)q^n\equiv4\psi(q)\pmod8.\]  The coefficient
of $q^n$ in $\psi(q)$ is $1$ if $n=r(r+1)/2$ for some $r\geq0$, that is, if
$8n+1=(2r+1)^2$ is a perfect square, and $0$ otherwise.  This gives us the first part of the result.

If $8n+1$ is a quadratic non-residue modulo some prime, it is
not a square.  Modulo $3$, $8n+1\equiv2$ exactly when $n\equiv2\pmod3$,
and $2(3n+2)+1=6n+5$.  Modulo $5$, the non-residues are $2,3$, and
$8n+1\equiv2,3\pmod5$ exactly when $n\equiv2,4\pmod5$; since
$2(5n+2)+1=10n+5$ and $2(5n+4)+1=10n+9$, this proves
\eqref{eq:cor-mod8}.
\end{proof}

\section{Proof of Theorem \ref{thm:ten}}\label{sec:ten}

Theorem~\ref{thm:general} with $k=10$, $m=3$ gives
\begin{equation}
 \sum_{n\ge0}\cphi_{10}(2n+1)q^n\equiv4\frac{f_4^4}{f_1^4}\CPhi_5(q)\pmod8,
 \label{eq:ten-odd}
\end{equation}
so it suffices to know $\CPhi_5(q)$ modulo $2$ to prove Theorem \ref{thm:ten}.

\begin{lemma}\label{lem:C5mod2}
We have
\begin{equation}
 \CPhi_5(q)\equiv\frac{1}{f_1f_4}\pmod2.
 \label{eq:C5mod2}
\end{equation}
\end{lemma}

\begin{proof}
By \eqref{eq:bin2m} with $m=2$, we have
$J(z,q)^4\equiv J(z^2,q^2)^2\pmod2$, so by \eqref{eq:Andrews-CT} and
\eqref{eq:J-theta}
\begin{align*}
\CPhi_5(q)=\CT_zJ(z,q)^5&\equiv\CT_z\bigl(J(z,q)J(z^2,q^2)^2\bigr)\\
 &=\frac{1}{f_1f_2^2}\CT_z\bigl(\Theta(z,q)\Theta(z^2,q^2)^2\bigr)\pmod2.
\end{align*}
By Lemma~\ref{lem:theta-square} with $q$ replaced by $q^2$ and $z$ by $z^2$,
\[
 \Theta(z^2,q^2)^2\equiv\varphi(q^2)\,\Theta(z^4,q^4)
 =\varphi(q^2)\sum_{r\in\Z}z^{4r}q^{2r(r+1)}\pmod2,
\]
which contains only powers $z^{4r}$.  Splitting $\Theta(z,q)$ according
to the parity of the exponent of $z$, we have
\[
 \Theta(z,q)=\sum_{s\in\Z}z^{2s}q^{s(2s+1)}+\sum_{s\in\Z}z^{2s+1}q^{(s+1)(2s+1)}.
\]
Only the first sum can contribute to the constant term of
$\Theta(z,q)\Theta(z^2,q^2)^2$, and it does so precisely when
$2s+4r=0$, that is, when $s=-2r$.  Hence
\begin{align*}
 \CT_z\bigl(\Theta(z,q)\Theta(z^2,q^2)^2\bigr)
 &\equiv\varphi(q^2)\sum_{r\in\Z}q^{(-2r)(-4r+1)}\,q^{2r(r+1)}\\
 &=\varphi(q^2)\sum_{r\in\Z}q^{10r^2}=\varphi(q^2)\varphi(q^{10})\pmod2.
\end{align*}
Thus, we have
\[
\CPhi_5(q)\equiv\frac{1}{f_1f_2^2}\varphi(q^2)\varphi(q^{10})\pmod2.
\]
Since $\varphi(q)=1+2\sum_{n\ge1}q^{n^2}\equiv1\pmod2$, we get
$\CPhi_5(q)\equiv1/(f_1f_2^2)\pmod2$, and $f_2^2\equiv f_4\pmod2$ gives
\eqref{eq:C5mod2}.
\end{proof}

\begin{proof}[Proof of Theorem~\ref{thm:ten}]
Substituting Lemma~\ref{lem:C5mod2} into \eqref{eq:ten-odd} (the factor
$4$ means only the residue modulo $2$ of the remaining series matters),
\[
 \sum_{n\ge0}\cphi_{10}(2n+1)q^n\equiv4\frac{f_4^4}{f_1^5f_4}=4\frac{f_4^3}{f_1^5}\pmod8.
\]
Modulo $2$, $f_1^4\equiv f_4$ and $f_4^2\equiv f_8$, so
$f_4^3/f_1^5=f_4^3/(f_1^4f_1)\equiv f_4^2/f_1\equiv f_8/f_1\pmod2$, which
gives \eqref{eq:ten}.
\end{proof}

\section{Concluding Remarks}\label{conc}

\begin{enumerate}
    \item Theorem~\ref{thm:odd18} determines $\cphi_{18}(2n+1)$ modulo $16$
completely, since the second term on the right of \eqref{eq:odd18} is
$8$ times a series whose coefficients only matter modulo $2$; using
$f_r^2\equiv f_{2r}\pmod2$ that series may be written as
$q^4\psi(q)f_{36}^3/f_{12}$.  For the even part, \eqref{eq:general-parts}
gives $\sum_{n\ge0}\cphi_{18}(2n)q^n\equiv\frac{f_2^{12}}{f_1^8f_4^4}\CPhi_9(q)\pmod{16}$,
so a complete description of $\cphi_{18}(n)$ modulo $16$ would follow from
one of $\CPhi_9(q)$ modulo $16$; we do not pursue this here.
\item The series $f_8/f_1$ is the generating function for partitions with no
part divisible by $8$, so Theorem~\ref{thm:ten} says that $\cphi_{10}(2n+1)\equiv4\pmod8$
exactly when the number of such partitions of $n$ is odd.  In the same
way, $\cphi_{26}(2n+1)$ modulo $8$ and $\cphi_{34}(2n+1)$ modulo $16$ can be
reduced to $\CPhi_{13}$ modulo $2$ and $\CPhi_{17}$ modulo $4$, for which
\cite[Theorems~3.6 and 3.7]{ChanWangYang2019} provide exact formulas; we
leave this to the interested reader.
\end{enumerate}

\section*{Acknowledgements}

The author used OpenAI's ChatGPT 5.6 Sol for the preparation of the initial version of this manuscript, and then used Anthropic's Claude Fable 5.1 for reviewing, which then suggested Corollary \ref{cor:sharp}. The author had a weaker version of Theorem \ref{thm:general} in that version, when we asked Fable 5.1 to search for a more general version of the result, based on computations, it gave the generalization in Theorem \ref{thm:general}. Language changes and some references were also suggested by Fable 5.1. The author has verified all of the mathematical content in this manuscript and takes full responsibility for it.

\end{document}